\documentclass{article}
\RequirePackage[leqno]{amsmath}[2.14]
\usepackage{amssymb}
\usepackage{amsthm}
\usepackage{tabularx}
\usepackage{picture,slashbox}
\usepackage{graphicx}
\usepackage{epsfig,epsf,psfrag,epstopdf,subfigure}
\usepackage{color}
\usepackage{comment}
\usepackage{url}
\usepackage{enumitem}

\RequirePackage{xr-hyper}[6.00]
\RequirePackage{ifpdf}[2.3]
\RequirePackage{hyperref}[6.83]
\RequirePackage{xcolor}[2.11]
\colorlet{inlinkcolor}{green!50!black}
\colorlet{exlinkcolor}{red!50!black}
\hypersetup{
  colorlinks = true,
  allcolors = inlinkcolor,
  urlcolor = exlinkcolor,
}

\RequirePackage[capitalize,nameinlink]{cleveref}[0.19]

\crefname{section}{section}{sections}
\crefname{subsection}{subsection}{subsections}
\Crefname{section}{Section}{Sections}
\Crefname{subsection}{Subsection}{Subsections}

\Crefname{figure}{Figure}{Figures}

\crefformat{equation}{\textup{#2(#1)#3}}
\crefrangeformat{equation}{\textup{#3(#1)#4--#5(#2)#6}}
\crefmultiformat{equation}{\textup{#2(#1)#3}}{ and \textup{#2(#1)#3}}
{, \textup{#2(#1)#3}}{, and \textup{#2(#1)#3}}
\crefrangemultiformat{equation}{\textup{#3(#1)#4--#5(#2)#6}}%
{ and \textup{#3(#1)#4--#5(#2)#6}}{, \textup{#3(#1)#4--#5(#2)#6}}{, and \textup{#3(#1)#4--#5(#2)#6}}

\Crefformat{equation}{#2Equation~\textup{(#1)}#3}
\Crefrangeformat{equation}{Equations~\textup{#3(#1)#4--#5(#2)#6}}
\Crefmultiformat{equation}{Equations~\textup{#2(#1)#3}}{ and \textup{#2(#1)#3}}
{, \textup{#2(#1)#3}}{, and \textup{#2(#1)#3}}
\Crefrangemultiformat{equation}{Equations~\textup{#3(#1)#4--#5(#2)#6}}%
{ and \textup{#3(#1)#4--#5(#2)#6}}{, \textup{#3(#1)#4--#5(#2)#6}}{, and \textup{#3(#1)#4--#5(#2)#6}}

\crefdefaultlabelformat{#2\textup{#1}#3}

\theoremstyle{plain}

\newtheorem{lemma}{Lemma}
\newtheorem{remark}{Remark}

\newtheorem{definition}{Definition}
\newtheorem{exmp}{Example}

\numberwithin{equation}{section}
\numberwithin{table}{section}
\numberwithin{figure}{section}
\numberwithin{theorem}{section}

\newcommand{\be}{\begin{equation}}
\newcommand{\ee}{\end{equation}}
\newcommand{\ba}{\begin{aligned}}
\newcommand{\ea}{\end{aligned}}
\newcommand{\bea}{\begin{eqnarray}}
\newcommand{\eea}{\end{eqnarray}}

\newcommand\cs{\mathcal{S}}
\newcommand\cd{\mathcal{D}}

\newcommand\bd{\mathbf{D}}

\newcommand\bi{\mathbf{I}}

\newcommand\bn{\mathbf{n}}
\newcommand\bny{\mathbf{n}(\mathbf{y})}

\newcommand\bp{\mathbf{p}}
\newcommand\br{\mathbf{r}}

\newcommand\bt{\mathbf{T}}

\newcommand\lbr{\|\br\|}

\newcommand\lbrt{\|\br\|^2}

\newcommand\bphi{\boldsymbol{\phi}}

\newcommand\bu{\mathbf{u}}
\newcommand\bv{\mathbf{v}}
\newcommand\bx{\mathbf{x}}

\newcommand\by{\mathbf{y}}

\newcommand\bxy{\bx-\by}

\newcommand\brt{\frac{\br \otimes \br}{\|\br\|^2}}
\newcommand\bgg{\mathbf{G}(\bx,t;\by,\tau)}
\newcommand\bpp{\bp(\bx,t;\by,\tau)}
\newcommand\xyt{(\bx,t;\by,\tau)}

\newcommand\ndr{\bny\cdot\br}
\newcommand\nxr{\bny\otimes\br}
\newcommand\rxn{\br\otimes\bny}

\newcommand\nxn{\bn\otimes\bn}
\newcommand\nxt{\bn\otimes\bt}

\newcommand\txt{\bt\otimes\bt}
\newcommand\txn{\bt\otimes\bn}

\newcommand\ttau{t-\tau}

\newcommand\dt{\Delta t}

\title{On the accurate evaluation of unsteady Stokes layer potentials
  in moving two-dimensional geometries}
\author{Leslie Greengard\thanks{Courant Institute of Mathematical 
    Sciences, New York University, New York, New York 10012 and
    Flatiron Institute, Simons Foundation, New York, New York 10010
    ({\tt greengard@courant.nyu.edu}).}
  \and
  Shidong Jiang\thanks{Department of Mathematical Sciences,
    New Jersey Institute of Technology, Newark, New Jersey 07102
    ({\tt shidong.jiang@njit.edu}). This
    research was supported by NSF under grant DMS-1720405 and
    by the Flatiron Institute, a division of the Simons Foundation.}
  \and
  Jun Wang\thanks{Flatiron Institute, Simons Foundation, New York, NY 10010
    ({\tt jwang@flatironinstitute.org}).}
}

\begin{document}

\maketitle

\begin{abstract}
  Two fundamental difficulties are encountered
  in the numerical evaluation of time-dependent layer potentials. One is the quadratic 
  cost of history dependence, which has been successfully addressed
  by splitting the potentials into 
  two parts - a local part that contains the most recent contributions and a history
  part that contains the contributions from all earlier times. 
  The history part is smooth, easily discretized using high-order
  quadratures, and straightforward to compute using a variety of fast algorithms.
  The local part, however, involves complicated singularities in the
  underlying  Green's function.  Existing methods, based on exchanging 
  the order of integration in space and time, are able to achieve high order
  accuracy, but are limited to the case of stationary boundaries. 
  Here, we present a new quadrature method 
  that leaves the order of integration unchanged, making use of 
  a change of variables that converts the singular integrals with respect to
  time into smooth ones. We have also derived asymptotic 
  formulas for the local part that lead to 
  fast and accurate hybrid schemes, extending earlier work for scalar heat potentials
  and applicable to moving boundaries.  The performance of the overall scheme 
  is demonstrated via numerical examples.
\end{abstract}


\section{Introduction} \label{introduction}

In this paper, we consider an integral equation approach to  
the linearized, incompressible Navier-Stokes equations
(also called {\em unsteady Stokes flow})
in a {\it nonstationary} domain $D_T = \prod_{\tau=0}^T D(\tau)$ with smooth boundary
$\Gamma_T=\prod_{\tau=0}^T\Gamma(\tau)$:
\begin{align}
  \frac{\partial {\bf u}}{\partial t}&=\Delta {\bf u}-\nabla p + {\bf g},
  \qquad (\bx,t) \in D_T, \label{stokeseq}\\
\nabla \cdot {\bf u}&=0,   \qquad (\bx,t) \in D_T, \label{dfcond}\\
{\bf u}(\bx,0) &= {\bf u}_0(\bx), \qquad \bx \in D(0), \label{stokeseqs:ic}
\end{align}
subject to either Dirichlet (``velocity") boundary conditions
\begin{equation}
{\bf u}(\bx,t) = {\bf f}(\bx,t), \qquad (\bx,t) \in \Gamma_T 
\label{stokeseqs:dbc}
\end{equation}
or Neumann (``traction") boundary conditions
\begin{equation}
\left(\frac{\partial \bu_i(\bx,t)}{\partial \bx_k}
+\frac{\partial \bu_k(\bx,t)}{\partial \bx_i}-p(\bx,t)\delta_{ik}\right) 
\bn_k(\bx)
=\mathbf{f}(\bx,t), \qquad (\bx,t) \in \Gamma_T.
\label{stokeseqs:nbc}
\end{equation}

While unsteady Stokes flow is of interest in its own right in modeling slow
viscous flow, with applications
in microfluidics \cite{karniadakis2005,kim2005}, it also arises in solving
the fully nonlinear incompressible Navier-Stokes
equations, where
${\bf g} = - {\bf u} \cdot \nabla {\bf u}$. 
In fact, most widely used marching schemes for the nonlinear problem
treat the advective term explicitly so that ${\bf g}(\bx,t)$ can be 
considered a known function when marching in time
\cite{ascher1995sinum,brown,chorin,henshaw,liuliupego,temam}.

We are interested here in methods for the unsteady Stokes equations that
enforce the divergence-free condition exactly, without the need for a projection
step. Recently, we described a ``mixed potential" method that accomplishes
this task through a Helmholtz decomposition of the forcing term ${\bf g}$
\cite{mixedpot}.
In the present paper, we continue our investigation, begun in
\cite{jiang2012sisc}, of integral equation methods that rely on the Green's 
function for the unsteady Stokes equations - the so-called unsteady
Stokeslet. We will discuss the relative merits of the mixed potential approach and
the unsteady Stokeslet-based approach in the concluding section.
For the moment, we simply note that
initial, volume, and layer potentials based on the unsteady Stokeslet 
involve nothing more than convolution with the Green's function without any
Helmholtz decomposition.
Moreover, standard velocity or traction boundary conditions can be imposed
using the double layer or single layer potential, respectively. These lead to 
well-conditioned integral equations of Volterra type.

For {\em stationary} boundaries, high-order accurate quadrature schemes
have also developed \cite{jiang2012sisc}, following the approach developed
for layer heat potentials in \cite{greengard2000acha,li2009sisc,lin1993thesis}. 
That is, the layer potentials are split into two parts -
a local part and a history part, where the local part contains the temporal
integration on the interval $[t-\delta,t]$ and the history part contains the temporal
integration on $[0,t-\delta]$. The local part involves essential singularities
in time, treated by exchanging the order of integration in space and time,
and carrying out product integration in time analytically. The history
part requires fast algorithms, but is more or less straightforward to discretize 
since the integrals encountered are smooth in time.

When the boundary is {\em nonstationary}, the aforementioned scheme can still be used
to evaluate heat layer potentials accurately. As observed in \cite{li2009sisc}, the
heat kernel admits the following factorization:
\be\label{heatkernel}
\ba
G_H(\bx,t;\by(\tau), \tau)&=\frac{1}{4\pi(t-\tau)}
e^{-\frac{\|\bx-\by(\tau)\|^2}{4(t-\tau)}}\\
&=\frac{1}{4\pi(t-\tau)}e^{-\frac{\|\bx(t)-\by(t)\|^2}{4(t-\tau)}}\cdot
e^{-\frac{\|\by(t)-\by(\tau)\|^2}{4(t-\tau)}}\cdot
e^{-\frac{(\bx(t)-\by(t))\cdot (\by(t)-\by(\tau))}{2(t-\tau)}}.
\ea
\ee
Note that the first term on the right side of \cref{heatkernel} can be dealt with
via product integration, as in the stationary case, while the second and third
terms are both smooth so long as the boundary motion is smooth,
since the factor $\frac{\by(t)-\by(\tau)}{(t-\tau)}$ is then 
well behaved as a function of $\tau$. 
Unfortunately, this simple modification
fails for unsteady Stokes layer potentials. The
unsteady Stokeslet $\bgg$
\cite{guenther2007jmfm,jiang2012sisc} is given by the formula
\begin{equation}\label{uskernel}
\begin{aligned}
\bgg&=\frac{e^{-\lbr^2/4 (\ttau)}}{4\pi (\ttau)}\left(\bi-\brt\right)
-\frac{1-e^{-\lbr^2/4 (\ttau)}}{2\pi \lbr^2} \left(\bi-2\brt\right),
\end{aligned}
\end{equation}
where $\br=\bxy$. As a result, when the boundary is moving, 
$\br=\bx-\by(\tau)$ and the first term can be handled as above but the second
term on the right-hand side of \cref{uskernel} cannot be factorized as a Stokeslet
on a fixed domain modulated by a smooth function, due to the presence of
the factor $\lbr^2$ in the denominator.

Here, we present an accurate numerical scheme for the evaluation of
the local part of the unsteady Stokes layer potentials for both static and
moving geometries. For this, we split the local part further into two parts:
$[t-\delta,t] = [t-\delta,t-\epsilon] \cup [t-\epsilon,t]$ - the second part
is treated asymptotically and the interval $[{t-\delta},{t-\epsilon}]$ is 
treated by a change of variables in the nearly singular integrals,
as in \cite{wang2018acom}.
We carry out the asymptotic analysis only to lowest
order for both the single and double layer potentials. The double layer
derivation is somewhat technical as compared with the 
double layer heat potential \cite{greengard1990cpam,wang2018acom} 
because the kernel is not Riemann-integrable and defined only in
the principal value sense. Furthermore,
although the first asymptotic term, of the order
$\sqrt{\epsilon}$, is local in space-time, the next term of order
$O(\epsilon)$ involves an integral on the entire spatial boundary.
By contrast, asymptotic expansions for heat layer potentials involves terms
which remain local (although they involve higher and higher order 
spatial derivatives for higher and higher powers of $\epsilon$.)

An important difference between the current approach and the earlier method of
\cite{jiang2012sisc} is that the spatial integrals are now singular rather than
weakly singular and have to be interpreted
in the principal value sense. Fortunately, there are many high-order rules
available, such as the Gauss-trapezoidal rule of \cite{alpert1999sisc}.
After combining all these tools, the overall scheme is high-order accurate 
even for {\it nonstationary} boundaries and the linear systems which arise from
implicit time-marching schemes are well-conditioned and amenable to
solution using iterative schemes such as GMRES \cite{gmres}.

The paper is organized as follows. In \cref{sec:prelim}, we state
some needed integral identities and summarize the relevant properties 
of single and double layer potentials for unsteady Stokes flow. 
In \cref{sec:asym}, we derive the leading order asymptotic expansions for
layer potentials and in \cref{sec:nearpart},
we discuss the numerical treatment of the nearly-singular parts.
In \cref{sec:alg}, a fully discrete numerical scheme is described for the
Dirichlet problem. Numerical examples are presented
in \cref{sec:examples} with some concluding remarks 
in \cref{sec:conclusion}.
\section{Mathematical preliminaries}\label{sec:prelim}
We turn now to a brief summary of potential theory for unsteady Stokes flow.
We refer the reader to \cite{fabes1977ajm1,fabes1977ajm2,shen1991ajm} for
a detailed analysis of 
the properties of these {\em parabolically singular} layer potentials.
\begin{definition}
Let $\bphi$ be a vector-valued function defined on $\Gamma_T$. 
Then the single layer potential operator $\cs$ is defined by the formula
\begin{equation}\label{slp}
  \cs[\bphi](\bx,t) = \int_0^t\int_{\Gamma(\tau)} \bgg \mathbf{\bphi}(\by,\tau)
  ds(\by)d\tau,
\end{equation}
where $\bgg$ is defined in \cref{uskernel}.
The double layer potential operator $\cd$ is defined by the formula
\begin{equation}\label{dlp}
\cd[\bphi](\bx,t) = \int_0^t\int_{\Gamma(\tau)} \bd\xyt 
\mathbf{\bphi}(\by,\tau)ds(\by)d\tau
+\int_{\Gamma(t)} \frac{\rxn}{2\pi\lbr^2}
\mathbf{\bphi}(\by,t)ds(\by),
\end{equation}
where
\begin{equation}\label{dlpkernel}
\begin{aligned}
  \bd\xyt&=\frac{\nxr+(\ndr)(\bi-2\brt)}{8\pi}\frac{e^{-\lambda}}{(\ttau)^2}\\
&-\frac{\nxr+\rxn+(\ndr)(\bi-4\brt)}{8\pi}
\frac{1-e^{-\lambda}-\lambda e^{-\lambda}}{\lambda^2 (\ttau)^2},
\end{aligned}
\end{equation}
with $\lambda=\frac{\lbr^2}{4(\ttau)}$.
The kernel in the second term of \cref{dlp}
is the contribution of the instantaneous {\em pressurelet}
$\bpp$, denoted by
\begin{equation}\label{pressurelet}
\bpp = \frac{\br}{2\pi\lbr^2}\delta(\ttau),
\end{equation}
where the Dirac $\delta$ function
is understood to satisfy the condition $\int_0^t \delta(\ttau)d\tau=1$.
\end{definition}
We decompose the single layer potential $\cs[\bphi]$ 
defined in \cref{slp} into two parts - a local part and a history part:
\begin{equation}\label{slpsplit}
  \cs[\bphi]
  = \cs_L[\bphi] + \cs_H[\bphi],
\end{equation}
where the local part is 
\begin{equation}\label{slplocal}
\cs_L[\bphi](\bx,t)=\int_{t-\delta}^t \int_{\Gamma(\tau)} 
\bgg\bphi(\by,\tau)
ds(\by)d\tau, 
\end{equation}
and the history part is 
\begin{equation}\label{slphist}
S_H[\bphi](\bx,t)=\int_0^{t-\delta} \int_{\Gamma(\tau)}
\bgg\bphi(\by,\tau)ds(\by) d\tau.
\end{equation}
It is convenient to split 
the double layer potential $\cd[\bphi]$ into three parts:
a local part $\cd_L[\bphi]$, a history part $\cd_H[\bphi]$,
and a pressure part $\cd_P[\bphi]$:
\begin{equation}\label{dlpsplit}
\cd[\bphi] = \cd_L[\bphi] + \cd_H[\bphi] + \cd_P[\bphi]
\end{equation}
with
\[
\begin{aligned}
\cd_L[\bphi] &:=\int_{t-\delta}^t \int_{\Gamma(\tau)}
\bd\xyt\bphi(\by,\tau)ds(\by)d\tau,\\
\cd_H[\bphi] &:= \int_0^{t-\delta} \int_{\Gamma(\tau)}
\bd\xyt\bphi(\by,\tau)ds(\by) d\tau,\\
\cd_P[\bphi] &:= \int_{\Gamma(t)} \frac{\rxn}{2\pi\lbr^2}
\mathbf{\bphi}(\by,t)ds(\by),
\end{aligned}
\]
where the first and third terms on the right side of \cref{dlpsplit}
are understood in the principal value sense. For both layer potentials,
the parameter $\delta$ will be chosen to be a constant multiple of whatever
time step
$\dt$ is being used in a time-marching scheme. 
The density $\bphi$,
will be represented by a piecewise polynomial approximation with respect to 
the time variable.  The degree of that approximation
determines the time order of accuracy of the numerical scheme 
\cite{jiang2012sisc}.

We will make use of the following integral identities:
\begin{align}
 &\int_{-\infty}^{\infty}e^{-z^2}dz =\sqrt{\pi},\quad 
 \int_{-\infty}^{\infty}z^2e^{-z^2}dz&=\frac{\sqrt{\pi}}{2}, \label{ii1}\\
 &\int_{-\infty}^{\infty}\frac{1-e^{-z^2}}{z^2}dz&=2\sqrt{\pi}, \label{ii3}\\
 &\int_{-\infty}^{\infty}\frac{1-e^{-z^2}-z^2e^{-z^2}}{z^4}dz&=\frac{2\sqrt{\pi}}{3}.
  \label{ii4}
\end{align}

The formulas \cref{ii1} are well-known. To prove \cref{ii3},
let $f(x)$ be defined by the formula
\be
f(x) =   \int_{-\infty}^{\infty}\frac{1-e^{-xz^2}}{z^2}dz.
\ee
It is easy to show that $f(x)$ is well defined for $x\in(0,+\infty)$ since the
integrand is bounded as $z\rightarrow 0$ and integrable as
$z\rightarrow \pm\infty$. Moreover, calculation shows that
$\lim_{x\rightarrow 0^+} f(x) = 0$ and
$f'(x) = \int_{-\infty}^{\infty}e^{-xz^2}dz= \frac{\sqrt{\pi}}{\sqrt{x}}$ for $x>0$.
Integrating $f'(x)$ from $0$ to $1$, we obtain \cref{ii3}.
Similarly, let $g(x)$ be defined by
\be
g(x)=\int_{-\infty}^{\infty}\frac{1-e^{-xz^2}-xz^2e^{-xz^2}}{z^4}dz.
\ee
Then $\lim_{x\rightarrow 0^+} g(x) = 0$ and
$g'(x) = \int_{-\infty}^{\infty}xe^{-xz^2}dz= \sqrt{\pi x}$ for $x>0$.
Integrating $g'(x)$ from $0$ to $1$, we obtain \cref{ii4}.

\begin{remark}
To solve the equations \cref{stokeseq} with velocity boundary 
conditions, we seek a representation of the solution of
the form 
\[ 
{\bf u}(\bx,t) = 
\cd[\bphi](\bx,t) + {\cal V}[{\bf g}](\bx,t),
\]
where
\[
{\cal V}[{\bf g}](\bx,t) = \int_0^t\int_{\Omega(\tau)} 
\bgg \bf{g}(\by,\tau) \, d\by \, d\tau.
\]
This satisfies the partial differential equation and divergence
condition bu construction \cite{jiang2012sisc}.
Imposition of the boundary condition \cref{stokeseqs:dbc} leads to the
boundary integral equation 
\begin{equation} 
\label{biedir}
-\frac12 \bphi(\bx,t) +
\cd^\ast[\bphi](\bx,t) = {\bf f}(\bx,t) - {\cal V}[{\bf g}](\bx,t)
\end{equation}
where $\cd^\ast[\bphi](\bx,t)$ denotes the double layer potential
defined in the principal value sense.
We are primarily interested here in the solution of this equation and the design
of suitable quadrature methods and will assume that the 
the volume source term and corresponding potential 
${\cal V}[{\bf g}](\bx,t)$ are absent for the sake of simplicity.
\end{remark}
\section{Asymptotic analysis of the local layer potentials}
\label{sec:asym}
While it is possible to treat the local parts of the single and
double layer potentials by quadrature techniques alone,
it will turn out to be more efficient to 
split them further in the form: 
\begin{equation}\label{slocalsplit}
  \cs_L[\bphi] = \cs_\epsilon[\bphi] + \cs_N[\bphi]
\end{equation}
and
\begin{equation}\label{dlocalsplit}
\cd_L[\bphi] = \cd_\epsilon[\bphi] + \cd_N[\bphi]
\end{equation}
where
\[
\begin{aligned}
\cs_\epsilon[\bphi] &:=\int_{t-\epsilon}^t \int_{\Gamma(\tau)}
\bgg\bphi(\by,\tau)ds(\by)d\tau,\\
\cs_N[\bphi] &:= \int_{t-\delta}^{t-\epsilon} \int_{\Gamma(\tau)}
\bgg\bphi(\by,\tau)ds(\by) d\tau, \\
\cd_\epsilon[\bphi] &:=\int_{t-\epsilon}^t \int_{\Gamma(\tau)}
\bd\xyt\bphi(\by,\tau)ds(\by)d\tau\\
\cd_N[\bphi] &:= \int_{t-\delta}^{t-\epsilon} \int_{\Gamma(\tau)}
\bd\xyt\bphi(\by,\tau)ds(\by) d\tau.
\end{aligned}
\]
The terms $\cs_\epsilon[\bphi]$ and $\cd_\epsilon[\bphi]$ will be treated
by asymptotic methods, with $\epsilon$ chosen to be sufficiently small to
satisfy a given error tolerance.

To carry out the analysis, let
the reference ``target" point be denoted by $\bx\in \Gamma(t)$.
The unit tangent vector, unit normal vector and signed curvature
at $\bx$ are denoted by $\bt$, $\bn$, and $\kappa$, respectively.
The velocity at $(\bx, t)$ is denoted by $\bv$. Assuming the
curve is parametrized in arclength $s$, starting from
$\bx$, the ``source" point $\by(s,\tau)\in \Gamma(\tau)$
has the following Taylor expansion in $s$ and $\ttau$:
\be
\by(s,\tau) = \bx + \bt s -\frac{1}{2}\bn \, \kappa s^2 - (\bv\cdot\bn)\bn (\ttau)
+\cdots.
\ee
\begin{lemma}
The leading order asymptotic expansion of the single layer potential is given by
\be\label{slpasym2}
\cs_\epsilon[\bphi](\bx,t) 
=\sqrt{\frac{\epsilon}{\pi}}\txt \bphi(\bx,t)+ O(\epsilon).
\ee
\end{lemma}

\begin{proof}
We first split the spatial integral in $\cs_\epsilon[\bphi]$ into two parts:
\be
\int_{\Gamma(\tau)} = \int_{\Gamma(\tau)\cap B_a^c(\bx)}
+ \int_{\Gamma(\tau)\cap B_a(\bx)},
\ee
where $B_a(\bx)$ is a ball of radius $a$ centered at $\bx$ and $B_a^c(\bx)$
is its complement in $\mathbb{R}^2$. Here, $a$ is a fixed small positive number.
Clearly, $\lbr$ is bounded away
from zero on $\Gamma(\tau)\cap B_a^c(\bx)$. Thus, the term
$\frac{e^{-\lbr^2/4 (\ttau)}}{4\pi (\ttau)} \rightarrow 0$ exponentially fast
as $\epsilon\rightarrow 0$ for $\tau \in (t-\epsilon, t)$, and the term
$\frac{1-e^{-\lbr^2/4 (\ttau)}}{2\pi \lbr^2}$ approaches
$\frac{1}{2\pi \lbr^2}$. Combining these two facts, we conclude that
\be
\int_{t-\epsilon}^t\int_{\Gamma(\tau)\cap B_a^c(\bx)} \sim O(\epsilon),
\ee
and hence,
\be\label{slpasym}
\cs_\epsilon[\bphi](\bx,t) = \int_{t-\epsilon}^t \int_{\Gamma(\tau)\cap B_a(\bx)}
\bgg\bphi(\by,\tau)ds(\by)d\tau + O(\epsilon).
\ee
In the following asymptotic estimates, we assume $s^2=O(\ttau)$.
\be\label{asym1}
\ba
\br &= \bx - \by = -\bt s +\frac{1}{2}\bn\kappa s^2
+(\bv\cdot\bn)\bn (\ttau)+\cdots,\\
\lbrt &= s^2+O(s^3), \qquad \brt = \txt + O(s), \qquad ds(\by)=(1+O(s^2))ds.\\ 
\ea
\ee
Substituting \cref{asym1} into \cref{slpasym}, we obtain
\be
\ba
&\cs_\epsilon[\bphi](\bx,t) = \int_{t-\epsilon}^t \int_{-a}^a
\left(\frac{e^{-s^2/4 (\ttau)+O(s)}}{4\pi (\ttau)}\left(\bi-\txt+O(s)\right)\right.\\
&\left.-\frac{1-e^{-s^2/4 (\ttau)+O(s)}}{2\pi s^2} \left(\bi-2\txt+O(s)\right)\right)
\bphi(\by,\tau)(1+O(s^2))dsd\tau + O(\epsilon).
\ea
\ee
The change of variables $z=\frac{s}{\sqrt{4(\ttau)}}$ and $u=\sqrt{4(\ttau)}$
gives $s=zu$, $\tau=t-u^2/4$,
$dsd\tau = -2(\ttau)dzdu$. Thus,
\be
\ba
&\cs_\epsilon[\bphi](\bx,t) = 2\int_{0}^{2\sqrt{\epsilon}} \int_{-\infty}^\infty
\left(\frac{e^{-z^2}}{4\pi}\left(\bi-\txt\right)\right.\\
&\left.-\frac{1-e^{-z^2}}{8\pi z^2} \left(\bi-2\txt\right)\right)
\bphi(\by(zu,t-u^2/4),t-u^2/4)dzdu + O(\epsilon).
\ea
\ee
Note that $\bx=\by(0,t)$.
Substituting \cref{ii1,ii3} into the above expression, we obtain
\be\label{slpasym2a}
\cs_\epsilon[\bphi](\bx,t) = 4\sqrt{\epsilon}
\left(
\frac{\sqrt{\pi}}{4\pi}\left(\bi-\txt\right)
-\frac{\sqrt{\pi}}{4\pi} \left(\bi-2\txt\right)\right)\bphi(\bx,t)
+ O(\epsilon), 
\ee
from which the result follows.
\end{proof}
\begin{lemma}
The leading order asymptotic expansion of the double layer potential is given by
\be\label{dlpasym}
\ba
\cd_\epsilon[\bphi](\bx,t) &= \int_{t-\epsilon}^t \int_{\Gamma(\tau)}
\bd\xyt\bphi(\by,\tau)ds(\by)d\tau\\
&=\sqrt{\frac{\epsilon}{\pi}}\left\{
\left(\frac{1}{6}\bv\cdot\bn+\frac{1}{2}\kappa\right)\bi
+\frac{1}{3}(\bv\cdot\bn) \txt  \right. \\
&\left. -\left(\frac{\bv\cdot\bn}{6} + \frac{3\kappa}{2} \right)\nxn\right\}
\bphi(\bx,t) \\
&+\sqrt{\frac{\epsilon}{\pi}}\left\{\nxt+2\txn\right\}\bphi_s(\bx,t)
+O(\epsilon).
\ea
\ee
\end{lemma}

\begin{proof}
Analysis of the double layer is more involved because of the fact that it is
defined only in the principal value sense.
We proceed by first expanding various needed quantities in terms of the arclength
parameter $s$:
\begin{align}
\br = \bx - \by = -\bt s +\frac{1}{2}\bn\kappa s^2
+ (\bv\cdot\bn)\bn (\ttau)
+O(s^3), \nonumber \\
\bn(\by) = \bn + \bt\kappa s + O(s^2),\nonumber \\
\label{asyms}
\bn(\by)\cdot \br =  -\frac{1}{2}\kappa s^2+(\bv\cdot \bn) (\ttau) +O(s^3), \\
\nxr = -\nxt s+ \left(\frac{1}{2}\nxn-\txt\right)\kappa s^2
+(\bv\cdot\bn)\nxn (\ttau)+O(s^3), \nonumber \\
\rxn = -\txn s+\left(\frac{1}{2}\nxn-\txt\right)\kappa s^2
+(\bv\cdot\bn)\nxn (\ttau)+O(s^3), \nonumber \\
\bphi(\by,\tau) = \bphi(\bx,t) + \bphi_s(\bx,t) s + O(s^2). \nonumber
\end{align}
In the last expression, $\bphi_s$ is the derivative of $\bphi$ with respect to 
arclength.
Using the same change of variables as for the single layer, namely 
$z=\frac{s}{\sqrt{4(\ttau)}}$ and $u=\sqrt{4(\ttau)}$, we have
\begin{align}
&\int_{t-\epsilon}^t \int_{-\infty}^\infty s^2\frac{e^{-s^2/4(\ttau)}}{(\ttau)^2}dsd\tau
  =8\int_0^{2\sqrt{\epsilon}}\int_{-\infty}^\infty z^2e^{-z^2}dzdu=8\sqrt{\pi\epsilon},
  \label{ii5}\\
&\int_{t-\epsilon}^t \int_{-\infty}^\infty (\ttau)\frac{e^{-s^2/4(\ttau)}}{(\ttau)^2}dsd\tau
  =2\int_0^{2\sqrt{\epsilon}}\int_{-\infty}^\infty e^{-z^2}dzdu=4\sqrt{\pi\epsilon},
  \label{ii6}\\
&\int_{t-\epsilon}^t \int_{-\infty}^\infty s^2\frac{1-e^{-s^2/4(\ttau)}-s^2/(4(\ttau))
  e^{-s^2/4(\ttau)}}{(s^2/(4(\ttau)))^2(\ttau)^2}dsd\tau
=16\sqrt{\pi\epsilon}, \label{ii7}\\
&\int_{t-\epsilon}^t \int_{-\infty}^\infty (\ttau)\frac{1-e^{-s^2/4(\ttau)}-s^2/(4(\ttau))e^{-s^2/4(\ttau)}}{(s^2/(4(\ttau)))^2(\ttau)^2}dsd\tau
=\frac{8}{3}\sqrt{\pi\epsilon}. \label{ii8}
\end{align}

The desired result follows from
combining \cref{dlpkernel,dlocalsplit,asym1,asyms,ii5,ii6,ii7,ii8} 
after simplifying the resulting expression.
\end{proof}
\section{Quadrature methods for the nearly singular parts}\label{sec:nearpart}

We now consider the evaluation of the nearly singular contributions
to the single and double layer potentials, 
$\cs_N[\bphi]$ and $\cd_N[\bphi]$.
Inspection of the kernels shows that we need to consider
the following terms which involve singularities in either space or time:
\be\label{nsterms}
\brt, \qquad \frac{e^{-\lambda}}{\ttau}, \qquad \frac{e^{-\lambda}}{(\ttau)^2},
\qquad \frac{1-e^{-\lambda}}{\lambda (\ttau)},
\qquad \frac{1-e^{-\lambda}-\lambda e^{-\lambda}}{\lambda^2 (\ttau)^2},
\ee
where $\lambda=\frac{\lbr^2}{4(\ttau)}$. 

Each entry in the tensor
product $\brt$ is bounded by $1$ but does not have a definite
limiting value as $\|\br\|\rightarrow 0$. In \cite{jiang2012sisc},
it was shown that by carrying out product integration in time first, 
the resulting spatial convolution kernels have logarithmic singularities
for which there are effective quadrature rules.
The full double layer kernel (including the pressurelet) involves non-integrable
singularities, so it is
critical to use quadrature rules that integrate functions
in the principal value sense as well. 
Alpert's  Gauss-trapezoidal rule for logarithmic singularities
\cite{alpert1999sisc} accomplishes both tasks with very high order accuracy
for discretizations based on equispaced points with respect to an underlying
parametrization of the curve $\Gamma(t)$.
For adaptive methods, based on representing the boundary
as the concatenation of boundary segments, 
a variety of other high-order rules are available
\cite{ggq1,kolm,ggq2,ggq3,helsing2,helsing3,kress}. In all cases, the spatial
quadrature rules avoid kernel evaluation at the singular point $\br={\bf 0}$ itself. 
Thus, we will assume that $\br \ne {\bf 0}$ in the subsequent discussion.

The remaining four terms in \cref{nsterms} involve singularities in time. 
We need to integrate these terms when multiplied by a smooth 
function of $\tau$
on the interval $[t-\delta,t-\epsilon]$. Assuming for simplicity that the
smooth function is constant, we follow the approach introduced for heat potentials
in \cite{wang2018acom} and apply
the change of variables $\ttau=e^z$. 
Assuming the smooth function is constant as a function of $\tau$,
we have
\be\label{nsintegrals}
\ba
\int_{t-\delta}^{t-\epsilon}\frac{e^{-\lambda}}{\ttau}d\tau &=
\int_{\ln \epsilon}^{\ln \delta} e^{-\frac{\lbr^2}{4}e^{-z}}dz,\\
\int_{t-\delta}^{t-\epsilon}\frac{e^{-\lambda}}{(\ttau)^2}d\tau &=
\int_{\ln \epsilon}^{\ln \delta} e^{-\frac{\lbr^2}{4}e^{-z}}e^{-z}dz,\\
\int_{t-\delta}^{t-\epsilon}\frac{1-e^{-\lambda}}{\lambda (\ttau)}d\tau &=
\int_{\ln \epsilon}^{\ln \delta}\frac{4}{\lbr^2}\left(1- e^{-\frac{\lbr^2}{4}e^{-z}}\right)e^{z}dz,\\
\int_{t-\delta}^{t-\epsilon}\frac{1-e^{-\lambda}-\lambda e^{-\lambda}}{\lambda^2 (\ttau)^2}d\tau &=
\int_{\ln \epsilon}^{\ln \delta}\frac{16}{\lbr^4}\left(1- e^{-\frac{\lbr^2}{4}e^{-z}}
-\frac{\lbr^2}{4}e^{-z}e^{-\frac{\lbr^2}{4}e^{-z}}\right)e^{z}dz.
\ea
\ee
Note that all of the integrals in
\cref{nsintegrals} are smooth in the new variable $z$
(even for $\br={\bf 0}$). 
Following \cite{wang2018acom}, in which only the first two integrals
above arise, we use Gauss-Legendre
quadrature on the interval $[\ln \epsilon,\ln \delta]$ 
to compute these integrals and the corresponding temporal integration
in both $\cs_N[\bphi]$ and $\cd_N[\bphi]$. A detailed analysis
of the discretization error is nontrivial 
even for the case of the scalar heat kernel.
It is shown in \cite{wang2018acom}, however, that the error in $n$-point
Gauss-Legendre quadrature for the single layer potential is of the order 
$O(\Delta t^k) + O\left( \log \left(\frac{\Delta t}{\epsilon} \right) \, f(n) \right)$,
where $f(n)$ is an exponentially decreasing function of $n$.
The first term accounts for the use of a k$th$ order accurate
approximation of the density in time.
The second term is more subtle. The order of accuracy is low with
respect to the time step but compensated for by permitting controllable
precision by increasing $n$.
Our numerical experiments are consistent with the estimate above, but
in practice, local error estimation based on the desired precision will
more efficiently determine the number of nodes required than {\em a priori} analysis.
Numerical experiments show that the
number of quadrature points needed is about $10 \log_{10} (1/\epsilon)$ 
to achieve a precision of $\epsilon$ for
$\lbr\in [0,1]$, assuming that $\lbr$ is a smoothly varying function of $\tau$.
It is likely that we could reduce the number of nodes needed by a more specialized
generalized Gaussian rule \cite{ggq1,ggq2,ggq3}.
\section{Numerical Implementation}\label{sec:alg}

We illustrate the use of our hybrid scheme
in solving the problem of unsteady Stokes flow with velocity boundary conditions. 
The procedure follows that in
\cite{jiang2012sisc}, and we refer the reader to that paper for a detailed discussion.
In short, for the history part $\cd_H[\bphi]$, we make use of a Fourier spectral 
approximation of the unsteady Stokeslet. This permits the use of the nonuniform FFT
and recurrence relations, which reduces the cost of evaluation to 
$O(NM \log M)$, where $N$ is the number of time steps and $M$ is the number
of points in the discretization of the boundary.
Because the kernel separates in both space and time in the Fourier basis,
moving boundaries pose no difficulty.
The local part $\cd_L[\bphi]$
is handled by the techniques outlined above in \cref{sec:asym,sec:nearpart}.
Because of the error in the asymptotic piece, it is convenient to set 
the cutoff parameter $\epsilon$ to the user-specified tolerance
$\varepsilon$. The near singular error is then controlled by the number of 
nodes in the near-singular part, which is if the order $O(\log (1/\varepsilon)$.
It is also possible to forego the use of asymptotics entirely and use the 
near singular quadrature on $[t-\delta, t = \varepsilon^2]$ with an error of the 
order $O(\varepsilon)$ from the truncation in time.
This increases the number of Gauss-Legendre nodes needed,
but could have advantages in terms of robustness and is useful
for numerical validation of the asymptotic estimate and for step-size control.
Finally, it was shown in \cite{jiang2012sisc} that
any implicit multistep semi-discretization scheme results in a system of
{\it second kind} integral equations at each time step,
even though the time-dependent Volterra integral equations
themselves are not of the second kind~\cite{fabes1977ajm1,shen1991ajm}. Thus,
iterative solution using  GMRES requires only a modest number of iterations
to solve the resulting linear system.
\section{Numerical Results}\label{sec:examples}
\begin{figure}[t]
\centering
\includegraphics[height=40mm]{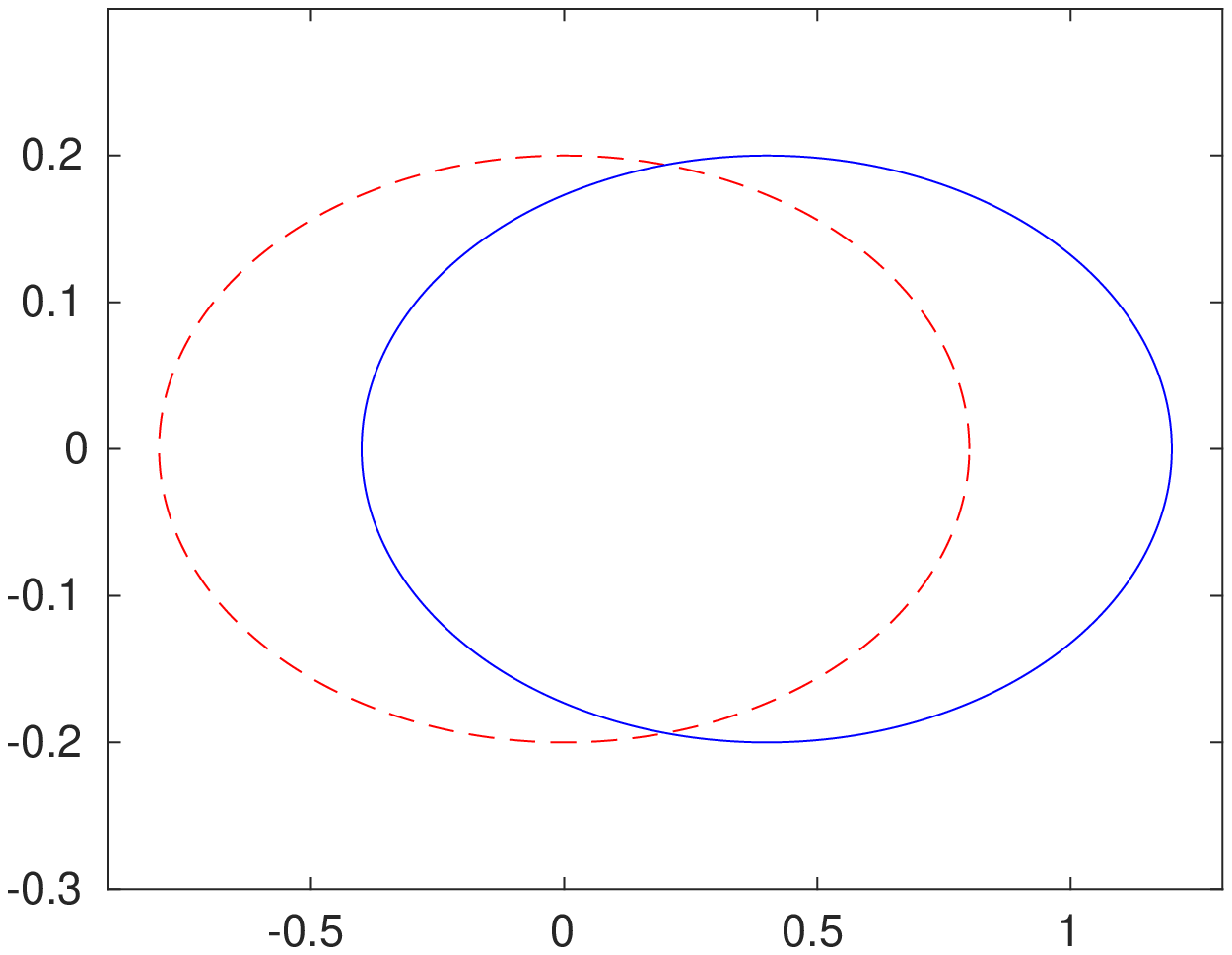}
\hspace{5mm}
\includegraphics[height=40mm]{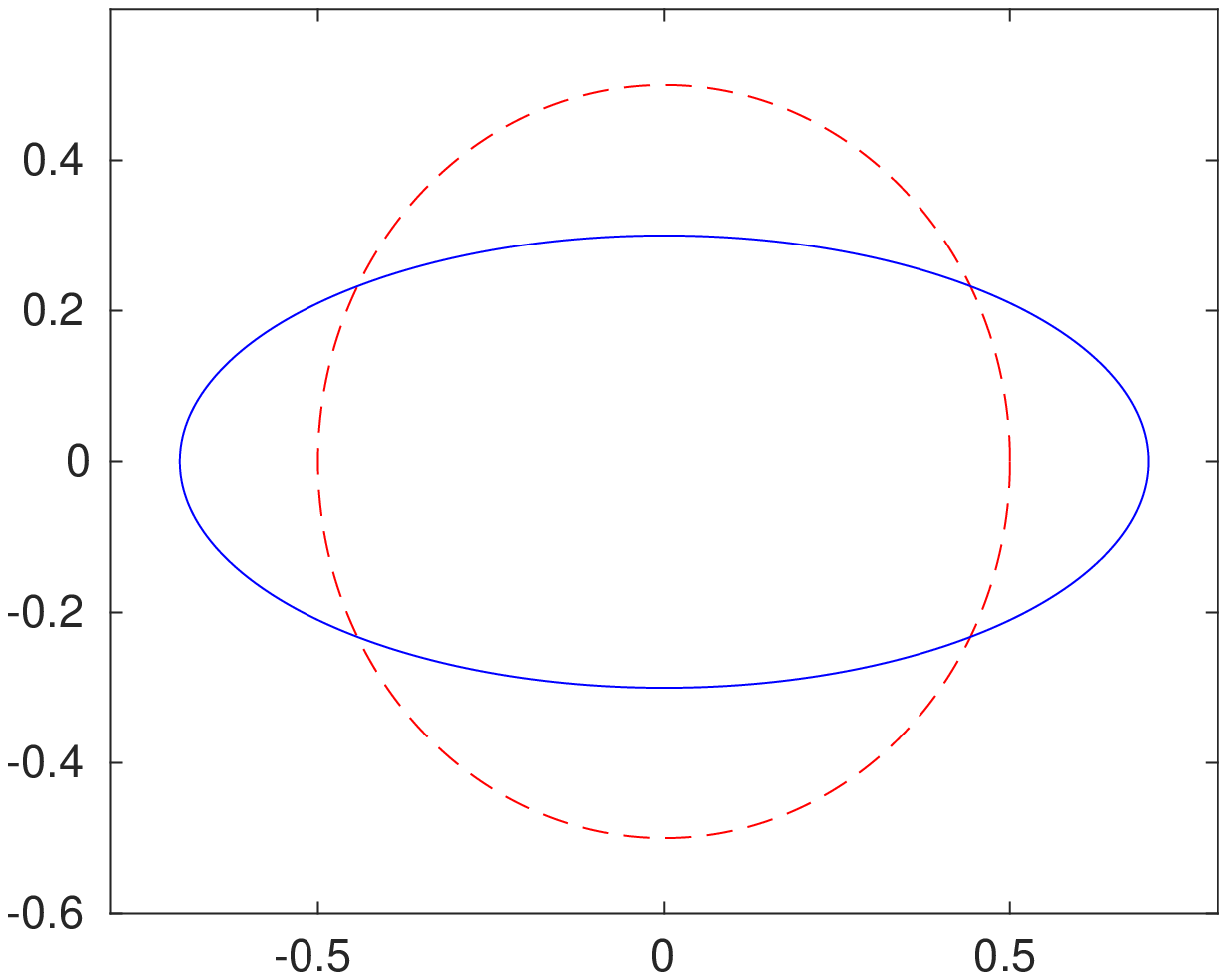}
\caption{Two moving boundaries. Left: An ellipse moving with constant speed.
  Right: a circle morphing into an ellipse. Red dashed line: initial position;
  blue solid line: final position.}
\label{fig1}
\end{figure}

We illustrate the performance of our method in two
{\it moving} geometries (\cref{fig1}):
\begin{enumerate}[label=(\alph*)]
\item an ellipse moving with constant speed
  \be\label{mb1}
  \left\{
  \ba
  y_1(\theta, t)&= 0.8\cos(\theta) + 0.4 t,\\
  y_2(\theta, t)&=0.2\sin(\theta),
  \ea
  \right. \qquad \theta\in [0, 2\pi]
  \ee
\item a circle deforming to an ellipse
  \be\label{mb2}
  \left\{
  \ba
  y_1(\theta, t)&= (0.5+0.2t)\cos(\theta),\\
  y_2(\theta, t)&= (0.5-0.2t)\sin(\theta),
  \ea
  \right. \qquad \theta\in [0, 2\pi],\ 0 \leq t \leq 1.
  \ee
\end{enumerate}
\begin{exmp} \label{exmp1}
\textsf{Validation of the asymptotic expansion \cref{slpasym2} of the unsteady Stokes 
single layer potential.}
\end{exmp}

To confirm the validity (and correctness) 
of the asymptotics for the
single layer potential, we calculate the single layer
potential on two moving boundaries for $t\in [0, T]$ with
$T=0.075$ and density function
\be
\bphi(\by,t)=\left(\cos(20y_2(\theta,T), 3y_1^3(\theta,T)\right).
\ee

A 12-digit accurate reference solution is computed using our near-singular 
quadrature rules on the interval $[0,T-\epsilon_M]$ with 16$th$ order accurate
spatial integration rules on a mesh with $200$ points. We then use
the hybrid asymptotic/numerical method and compute the near-singular part
on $[0,T-\epsilon]$ to 12-digit accuracy for various values of $\epsilon$.
After adding the asymptotic contribution, this should match the reference solution
with an error dominated by the asymptotics.
\Cref{fig2} shows the relative $l^2$ error as a function of
$\epsilon$ for the two moving boundaries in \cref{fig1}, which is 
clearly consistent with our analysis showing that it  
should be proportional to $\epsilon$. 
\begin{figure}[ht]
\centering
\includegraphics[height=40mm]{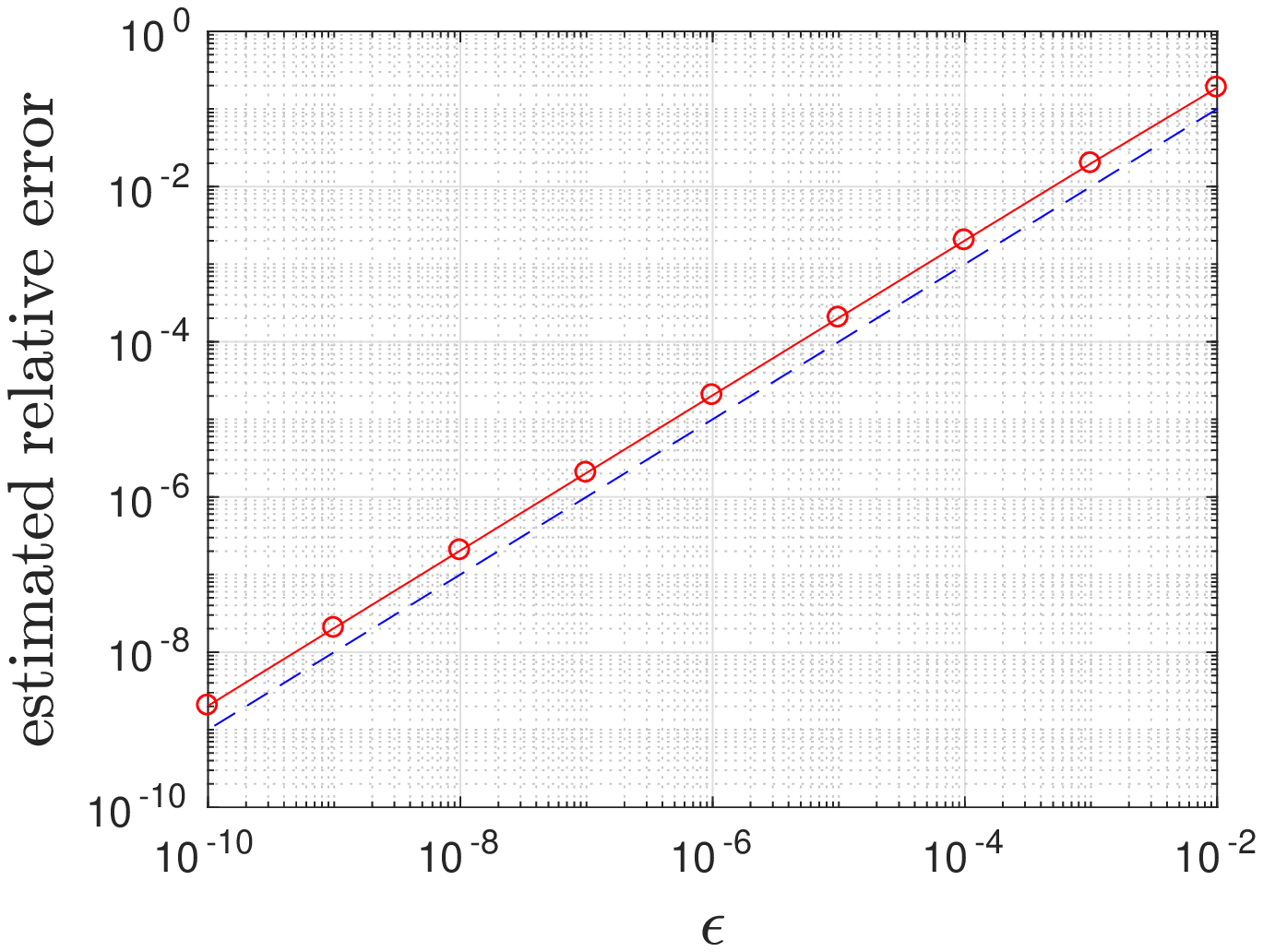}
\hspace{5mm}
\includegraphics[height=40mm]{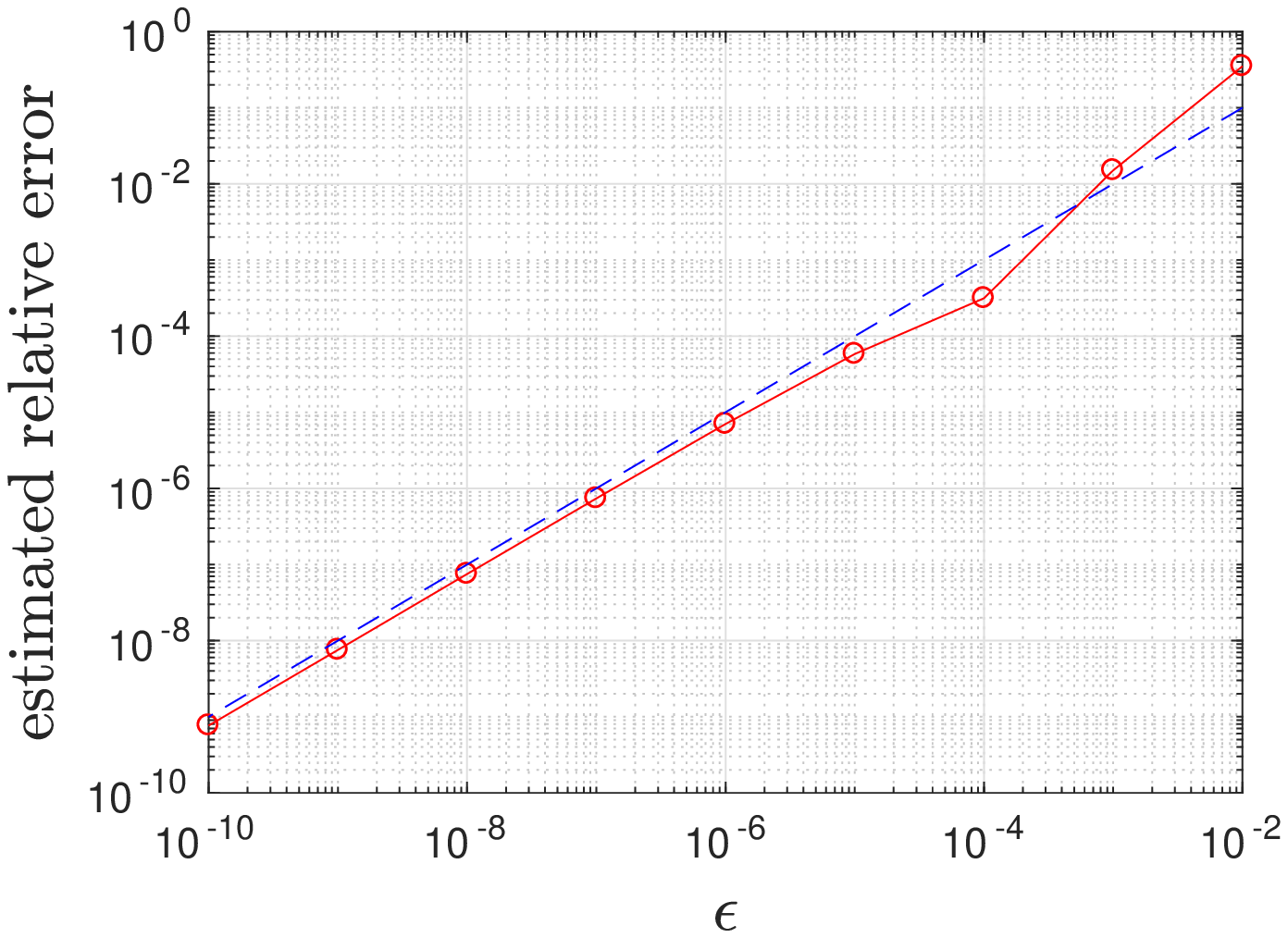}
\caption{Asymptotic errors in the single layer potential \cref{slpasym2}. 
  Red circles indicate numerical results and the dashed blue line is the function $y=10x$.
  Results for the moving boundary \cref{mb1} are plotted on the left and 
  results for the moving boundary \cref{mb2} are plotted on the right.}
\label{fig2}
\end{figure}

\begin{exmp} \label{exmp2}
\textsf{Validation of the asymptotic expansion \cref{dlpasym} of the unsteady Stokes 
double layer potential.}
\end{exmp}

In our second experiment,  we carry out the same analysis for the
double layer potential, with the same strategy 
for validation. The results are shown in \cref{fig3}, clearly showing the linear
decrease of the error with $\epsilon$.
\begin{figure}[ht]
\centering
\includegraphics[height=40mm]{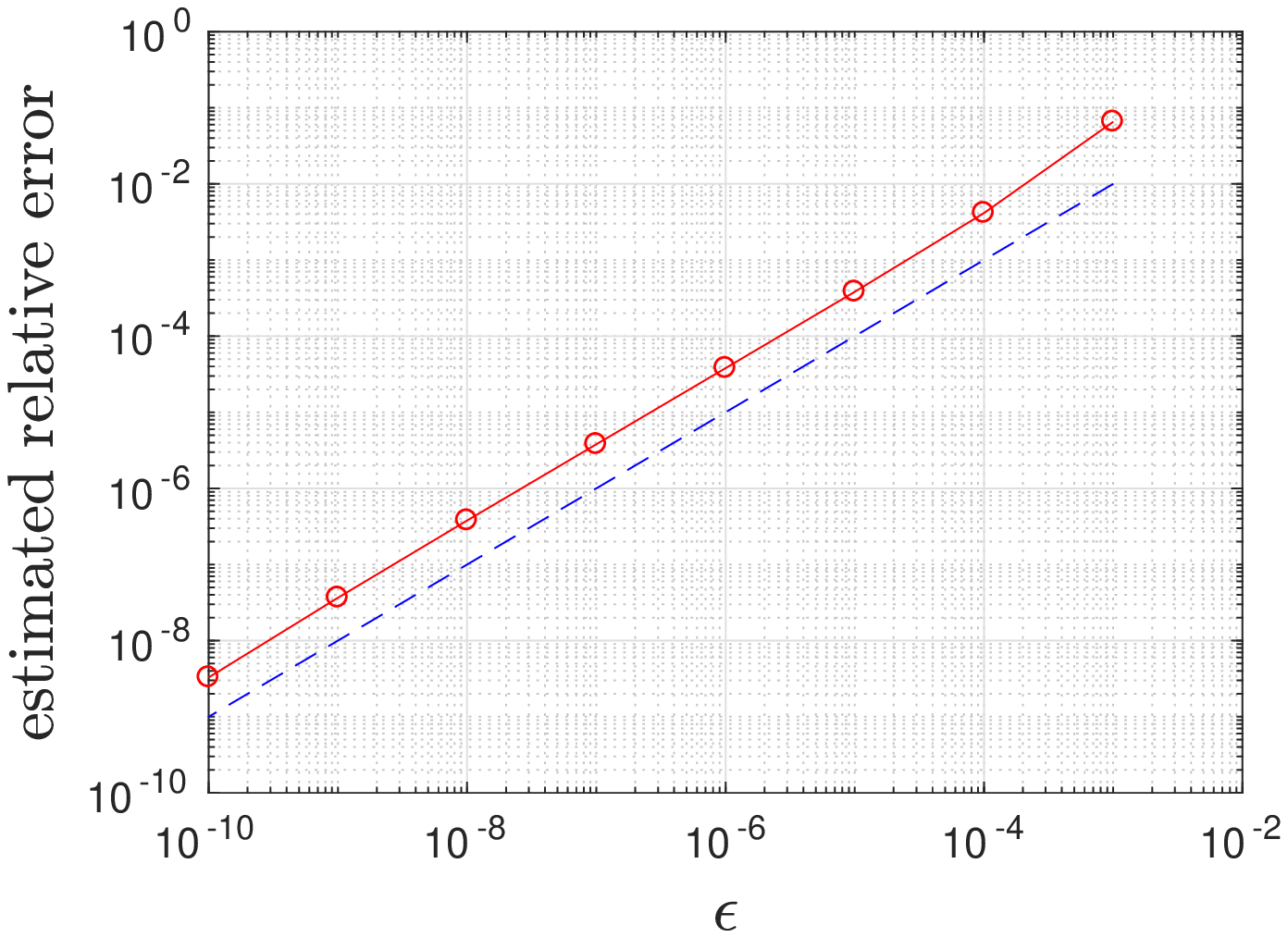}
\hspace{5mm}
\includegraphics[height=40mm]{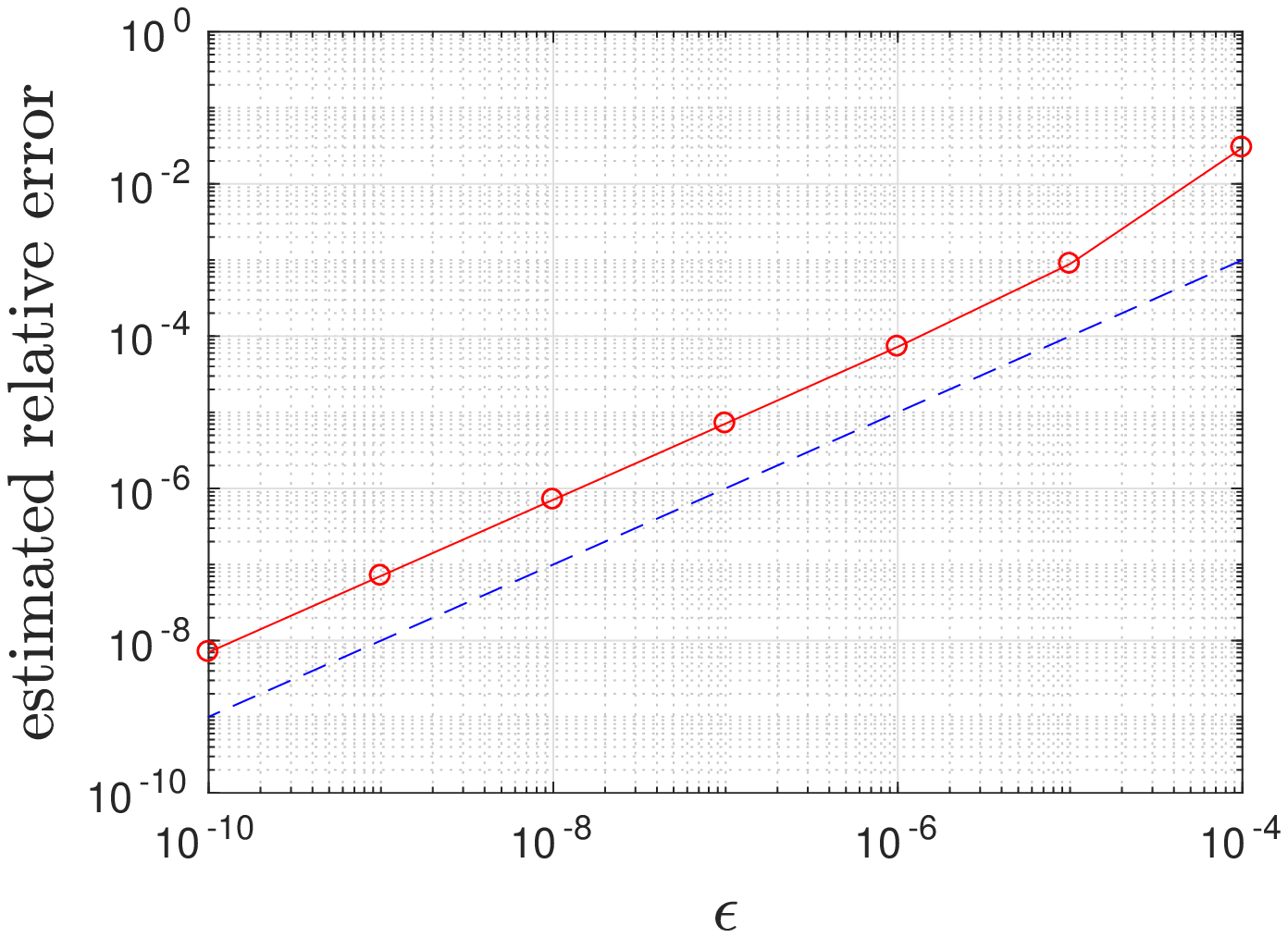}
\caption{Asymptotic errors in the double layer potential \cref{dlpasym}.
  Red circles indicate numerical results and the dashed blue line is the function $y=10x$.
  Results for the moving boundary \cref{mb1} are plotted on the left and 
  results for the moving boundary \cref{mb2} are plotted on the right.}
\label{fig3}
\end{figure}

\begin{exmp} \label{exmp3}
  \textsf{Unsteady Stokes flow for a bounded, moving
    domain with velocity boundary conditions.}
\end{exmp}

In our last example, 
we demonstrate the overall convergence of the full scheme for 
unsteady Stokes flow in a moving geometry, solving the 
integral equation \cref{biedir} (in the absence of a forcing term).
We use a scheme that is fourth order accurate in time and $16$th order in space,
with an exact solution of the form
\be
\ba
{\bf u}(\bx,t) &= \sum_{j=1}^{5}
\frac{1} {t^2}e^{-|\bx-\by_j|^2/(4t)}\left(x_2-y_{j2},-(x_1-y_{j1})\right)\\
&\quad +e^{-1/t}\sin(20t)e^{x_1}\left(\cos(x_2),-\sin(x_2)\right),
\ea
\ee
where ${\bf x} = (x_1,x_2)$ and the $\{ \by_j \}$ are chosen to be 
equispaced on the circle of radius $2$ centered at the origin,
which encloses both domains of interest. The number of spatial discretization
points is $200$ and the spatial discretization error is negligible. We place
$100$ test points inside the computational domain. \Cref{tab1} lists
the numerical results for both moving boundaries. Here $N$ is
the total number of time steps, $\dt$ is the step size, $E$ is the relative
$l^2$ error at the final time $T=1$, and $r$ is the ratio of relative $l^2$
errors for successive time step refinements. That is, $r(j) = E(j- 1)/E(j)$, which
gives a rough estimate of the convergence order. Note that $r \approx 2^4$,
consistent with fourth order convergence
once $\dt$ is sufficiently small.
      
\begin{table}[ht]
  \begin{center}
    {\bf \caption{Numerical results for solving the problem
        of unsteady Stokes flow with velocity boundary conditions
        in the moving geometries shown in \cref{fig1}.
\label{tab1}
}}
\begin{tabular*}{1.0\textwidth}{@{\extracolsep{\fill}} |c| c| c| c| c| c| c|  }
  \hline
  $N$ & $20$ & $40$ & $80$ & $160$ & $320$ & $640$\\
  \hline
  $\dt$ & $1/20$ & $1/40$ & $1/80$ & $1/160$ & $1/320$ & $1/640$ \\
  \hline\hline
  \multicolumn{7}{|c|}{Moving boundary \cref{mb1}}\\
  \hline
  E & $5.6\cdot 10^{-5}$ & $5.3\cdot 10^{-7}$ & $1.7\cdot 10^{-7}$
  & $9.9\cdot 10^{-9}$ & $5.8\cdot 10^{-10}$ & $3.8\cdot 10^{-11}$ \\
  \hline
  r & & $107$&  $3.0$&  $17.5$& $17.2$ & $15.4$\\
  \hline\hline
  \multicolumn{7}{|c|}{Moving boundary \cref{mb2}}\\
  \hline
  E & $4.7\cdot 10^{-5}$ & $3.6\cdot 10^{-7}$ & $1.6\cdot 10^{-7}$
  & $9.8\cdot 10^{-9}$ & $6.0\cdot 10^{-10}$ & $3.8\cdot 10^{-11}$ \\
  \hline
  r & &  $130$&  $2.3$&  $16$& $16$ & $15.8$\\
  \hline
\end{tabular*}
\end{center}
\end{table}

\section{Conclusions}\label{sec:conclusion}

We have developed a new method for the accurate evaluation of 
unsteady Stokes layer potentials in moving geometries.
The scheme is based on splitting the local parts of the layer potentials
into asymptotic and nearly-singular components.
The leading order asymptotic contributions are derived analytically and the 
nearly-singular parts are handled accurately via a single Gauss-Legendre 
quadrature panel using an exponential change of variables in time. 
Numerical experiments demonstrate that the scheme converges at the expected 
rate for flows in bounded domains with velocity boundary conditions.
One limitation of the current scheme is that the history part is handled
using a spectral approximation of the Green's function 
\cite{greengard2000acha,greengard1990cpam,lin1993thesis}. 
We are currently working on a marching scheme that represents the history 
part on an adaptive spatial mesh using the ``bootstrapping" method of 
\cite{wang2017thesis}.

It is worth noting that the recently developed
mixed potential method for unsteady Stokes flow \cite{mixedpot}
also permits high order accurate marching schemes in moving geometries.
An advantage of that method is that it requires only harmonic and layer 
potentials, simplifying the fast algorithm and quadrature issues.
A disadvantage is that it requires computation of the Helmholtz decomposition of the 
volume forcing term.
Unsteady Stokes potentials lead to better-conditioned integral equations when using
fully implicit marching schemes (at least for large time steps) and require 
only integration of the volume forcing term against the Green's function.
We intend to explore the relative performance of these two approaches 
in future work.

\bibliographystyle{abbrv}
\bibliography{journalnames,refs}

\begin{thebibliography}{10}

\bibitem{alpert1999sisc}
B.~K. Alpert.
\newblock Hybrid {G}auss-trapezoidal quadrature rules.
\newblock {\em SIAM J. Sci. Comput.}, 20(5):1551--1584, 1999.

\bibitem{ascher1995sinum}
U.~M. Ascher, S.~J. Ruuth, and B.~M. Wetton.
\newblock Implicit-explicit methods for time-dependent partial differential
  equations.
\newblock {\em SIAM J. Numer. Anal.}, 32:797--823, 1995.

\bibitem{ggq1}
J.~Bremer, Z.~Gimbutas, and V.~Rokhlin.
\newblock A nonlinear optimization procedure for generalized {G}aussian
  quadratures.
\newblock {\em SIAM J. Sci. Comput.}, 32(4):1761--1788, 2010.

\bibitem{brown}
D.~L. Brown, R.~Cortez, and M.~L. Minion.
\newblock Accurate projection methods for the incompressible {N}avier-{S}tokes
  equations.
\newblock {\em J. Comput. Phys.}, 168(2):464--499, 2001.

\bibitem{chorin}
A.~J. Chorin.
\newblock Numerical solution of the {N}avier-{S}tokes equations.
\newblock {\em Math. Comput.}, 22:745--762, 1968.

\bibitem{fabes1977ajm1}
E.~B. Fabes, J.~E. Lewis, and N.~M. Riviere.
\newblock Boundary value problems for the {N}avier-{S}tokes equations.
\newblock {\em Am. J. Math.}, 99:626--668, 1977.

\bibitem{fabes1977ajm2}
E.~B. Fabes, J.~E. Lewis, and N.~M. Riviere.
\newblock Singular integrals and hydrodynamic potentials.
\newblock {\em Am. J. Math.}, 99:601--625, 1977.

\bibitem{mixedpot}
L.~Greengard and S.~Jiang.
\newblock A new mixed potential representation for the equations of unsteady,
  incompressible flow.
\newblock {\em arXiv:1809.08442}, 2018.

\bibitem{greengard2000acha}
L.~Greengard and P.~Lin.
\newblock Spectral approximation of the free-space heat kernel.
\newblock {\em Appl. Comput. Harmon. Anal.}, 9:83--97, 2000.

\bibitem{greengard1990cpam}
L.~Greengard and J.~Strain.
\newblock A fast algorithm for the evaluation of heat potentials.
\newblock {\em Comm. Pure Appl. Math.}, 43:949--963, 1990.

\bibitem{guenther2007jmfm}
R.~B. Guenther and E.~A. Thomann.
\newblock Fundamental solutions of {S}tokes and {O}seen problem in two spatial
  dimensions.
\newblock {\em J. Math. Fluid Mech.}, 9:489--505, 2007.

\bibitem{helsing2}
J.~Helsing.
\newblock A fast and stable solver for singular integral equations on piecewise
  smooth curves.
\newblock {\em SIAM J. Sci. Comput.}, 33(1):153--174, 2011.

\bibitem{helsing3}
J.~Helsing and R.~Ojala.
\newblock Corner singularities for elliptic problems: integral equations,
  graded meshes, quadrature, and compressed inverse preconditioning.
\newblock {\em J. Comput. Phys.}, 227(20):8820--8840, 2008.

\bibitem{henshaw}
W.~D. Henshaw.
\newblock A fourth-order accurate method for the incompressible
  {N}avier-{S}tokes equations on overlapping grids.
\newblock {\em J. Comput. Phys.}, 113:13--25, 1994.

\bibitem{jiang2012sisc}
S.~Jiang, S.~Veerapaneni, and L.~Greengard.
\newblock Integral equation methods for unsteady {S}tokes flow in two
  dimensions.
\newblock {\em SIAM J. Sci. Comput.}, 34(4):A2197--A2219, 2012.

\bibitem{karniadakis2005}
G.~E. Karniadakis, A.~Beskok, and N.~Aluru.
\newblock {\em Microflows and Nanoflows}.
\newblock Springer, New York, 2005.

\bibitem{kim2005}
S.~Kim and S.~J. Karrila.
\newblock {\em Microhydrodynamics: {P}rinciples and {S}elected {A}pplications}.
\newblock Dover, New York, 2005.

\bibitem{kolm}
P.~Kolm and V.~Rokhlin.
\newblock Numerical quadratures for singular and hypersingular integrals.
\newblock {\em Comput. Math. Appl.}, 41(3--4):327--352, 2001.

\bibitem{kress}
R.~Kress.
\newblock {\em Linear Integral Equations}, volume~82 of {\em Applied
  Mathematical Sciences}.
\newblock Springer--Verlag, Berlin, third edition, 2014.

\bibitem{li2009sisc}
J.~Li and L.~Greengard.
\newblock High order accurate methods for the evaluation of layer heat
  potentials.
\newblock {\em SIAM J. Sci. Comput.}, 31:3847--3860, 2009.

\bibitem{lin1993thesis}
P.~Lin.
\newblock {\em On the Numerical Solution of the Heat Equation in Unbounded
  Domains}.
\newblock PhD thesis, Courant Institute of Mathematical Sciences, New York
  University, New York, 1993.

\bibitem{liuliupego}
J.-G. Liu, J.~Liu, and R.~L. Pego.
\newblock Stable and accurate pressure approximation for unsteady
  incompressible viscous flow.
\newblock {\em J. Comput. Phys.}, 229(9):3428--3453, 2010.

\bibitem{ggq2}
J.~Ma, V.~Rokhlin, and S.~Wandzura.
\newblock Generalized {G}aussian quadrature rules for systems of arbitrary
  functions.
\newblock {\em SIAM J. Numer. Anal.}, 33(3):971--996, 1996.

\bibitem{gmres}
Y.~Saad and M.~H. Schultz.
\newblock G{MRES}: a generalized minimal residual algorithm for solving
  nonsymmetric linear systems.
\newblock {\em SIAM J. Sci. Statist. Comput.}, 7(3):856--869, 1986.

\bibitem{shen1991ajm}
Z.~Shen.
\newblock Boundary value problems for parabolic {L}am\'{e} systems and a
  nonstationary linearized system of {N}avier-{S}tokes equations in {L}ipschitz
  cylinders.
\newblock {\em Am. J. Math.}, 113:293--373, 1991.

\bibitem{temam}
R.~Temam.
\newblock Sur l’approximation de la solution des equations de
  {N}avier-{S}tokes par la methode des fractionnarires {II}.
\newblock {\em Arch. Rational Mech. Anal.}, 33:377--385, 1969.

\bibitem{wang2017thesis}
J.~Wang.
\newblock {\em Integral equation methods for the heat equation in moving
  geometry}.
\newblock PhD thesis, Courant Institute of Mathematical Sciences, New York
  University, New York, September 2017.

\bibitem{wang2018acom}
J.~Wang and L.~Greengard.
\newblock Hybrid asymptotic/numerical methods for the evaluation of layer heat
  potentials in two dimensions.
\newblock {\em Adv. Comput. Math.}, accepted, 2018.

\bibitem{ggq3}
N.~Yarvin and V.~Rokhlin.
\newblock Generalized {G}aussian quadratures and singular value decompositions
  of integral operators.
\newblock {\em SIAM J. Sci. Comput.}, 20(2):699--718, 1998.

\end{thebibliography}

\end{document}